\documentclass[12pt]{amsart}
\usepackage{amsmath}
\allowdisplaybreaks
\usepackage{amssymb,latexsym}
\usepackage[matrix,arrow]{xy}
\usepackage[usenames]{color}
\usepackage{slashed}
\usepackage{parskip}
\usepackage{amsmath,amssymb,graphics}

\usepackage{graphicx}

\newcommand{\mathsym}[1]{{}}

\newtheorem{thm}{Theorem}[section]

\newtheorem{lem}[thm]{Lemma}
\newtheorem{prop}[thm]{Proposition}

\theoremstyle{definition}
\newtheorem{defn}{Definition}[section]
\numberwithin{equation}{section}
\theoremstyle{remark}

\theoremstyle{example}

\newcommand{\n}{\nabla}

\newcommand{\no }{\nonumber }

\newcommand{\md }{\mathrm{d}}

\newcommand{\de}{\partial}
\newcommand{\ov}{\overline}

\newcommand{\be}{\begin{equation}}
\newcommand{\ee}{\end{equation}}
\newcommand{\bag}{\begin{eqnarray}}
\newcommand{\eag}{\end{eqnarray}}
\newcommand{\ban}{\begin{eqnarray*}}
\newcommand{\ean}{\end{eqnarray*}}
\newcommand{\ba}{\begin{aligned}}
\newcommand{\ea}{\end{aligned}}
\newcommand{\lf}{\left}
\newcommand{\rt}{\right}

\newcommand{\al}{\alpha}

\newcommand{\bpf}{\begin{proof} }
\newcommand{\epf}{\end{proof} }

\newcommand{\ric}{\mathrm{Ric}}
\newcommand{\tr}{\mathrm{tr}}
\newcommand{\F}{\mathcal{F}}
\newcommand{\ti}{\tilde}
\newcommand{\mn}{\sqrt{-1}}
\newcommand{\dbar}{\overline{\partial}}
\newcommand{\ddt}{\frac{\partial}{\partial t}}
\newcommand{\ddb}{\sqrt{-1}\partial \ov{\partial}}

\newcommand{\R}{ \mathbb{R}}
\newcommand{\q}{ \mathbb{Q}}

\newcommand{\C}{ \mathbb{C}}
\newcommand{\h}{ \mathbb{H}}
\newcommand{\torus}{ \mathbb{T}}
\newcommand{\s}{ \mathbb{S}}

\begin{document}
\title{the Chern-Ricci flow on Oeljeklaus-Toma manifolds}
\author{Tao Zheng }
\thanks{Supported by China Postdoctoral Science Foundation funded project grant Nos. 2014M550620 and 2015T80040, and by  National Natural Science Foundation of China grant No. 11401023}
\subjclass[2010]{53C44, 53C55, 32W20, 32J18, 32M17}
\keywords{the Chern-Ricci flow, Oeljeklaus-Toma manifold, Calabi-type estimate, Gromov-Hausdorff convergence}

\maketitle
\begin{abstract}
We study the Chern-Ricci flow, an evolution equation of Hermitian metrics, on a family of Oeljeklaus-Toma (OT-) manifolds which are non-K\"{a}hler compact complex manifolds with negative Kodaira dimension. We prove that, after an initial conformal change, the flow converges, in the
Gromov-Hausdorff sense, to a torus with a flat Riemannian metric determined by the OT-manifolds themselves.
\end{abstract}

\section{Introduction}\label{section1}
The Chern-Ricci flow is an evolution equation for Hermitian metrics by their Chern-Ricci forms on complex manifolds, which coincides exactly with the K\"{a}hler-Ricci flow when the initial metric is K\"{a}hlerian. It was introduced by Gill \cite{gill} in the setting of complex manifolds with vanishing first Bott-Chern class. Tosatti and Weinkove \cite{twjdg,twcomplexsurface} investigated the flow on more general complex manifolds and proposed a program to study its behavior on all compact surfaces. The results in \cite{twjdg,twcomplexsurface,twymathann,gill13,gillsmith,nie,shermanweinkove} are closely similar to those for the K\"{a}hler-Ricci flow, and provide affirmative evidence that the Chern-Ricci flow is a natural geometric flow on complex surfaces whose properties reflect the underlying geometry of these manifolds.

Class \uppercase\expandafter{\romannumeral7} surfaces are by definition non-K\"{a}hler compact complex surfaces with negative Kodaira dimension and first Betti number one. This class of surfaces are of especial interest because there exists a well-known problem to complete their classification. Naturally, we will try to understand the properties of the Chern-Ricci flow on these surfaces, with the long-term aim of obtaining more topological or complex-geometric properties (cf. \cite{streetstian}, where a different flow is considered). In this direction, in \cite{ftwz}, the authors consider a family of Class \uppercase\expandafter{\romannumeral7} surfaces, known as Inoue surfaces (see \cite{inoue}), and proved that, for a large class of Hermitian metrics, the Chern-Ricci flow always collapses the Inoue surface to a circle at infinite time, in the Gromov-Hausdorff sense. Also, the authors \cite{ftwz} posed some conjectures and open problems concerned by the Chern-Ricci flow.

In this paper, we will concentrate on part of Problem 3 proposed in \cite {ftwz}, that is, we will study the behavior of the Chern-Ricci flow on a family of well-understood Oeljeklaus-Toma (OT-) manifolds, analogous to the Inoue-Bombieri surfaces $S_M$ (see \cite{inoue}) in high dimensions. OT-manifolds, first constructed by Oeljeklaus and Toma \cite {ot} from the view of algebraic number theory, are non-K\"{a}hler compact complex solvmanifolds with negative Kodaira dimension and without Vaisman metrics (see \cite[Proposition 2.5]{ot} and \cite[Section 6]{kasuya}). Battisti and Oeljeklaus \cite[Theorem 3.5]{bo} states that OT-manifolds admit no analytic hypersufaces and algebraic dimension is zero. Verbitsky \cite{verbitsky1,verbitsky2} also proved that OT-manifolds carry no closed $1$-dimensional analytic subspaces and that
OT-manifolds can not contain any nontrivial compact complex $2$-dimensional submanifolds except the Inoue surfaces.
More recent progresses and open problems about OT-manifolds can be found in \cite{ovu} and references therein.

We investigate a class of complex $m$-dimensional OT-manifolds denoted by $M_K$ with universal cover $\h^{m-1}\times \C$ and quotient covering map is denoted by $\pi:\;\h^{m-1}\times \C\longrightarrow M_K$ (see Section \ref{sectionot}), where $\h$ is the upper half plane. This class of OT-manifolds have locally conformally K\"{a}hler metric structure and admit no non-trivial complex subvariety (see \cite{ot,ov} and \cite[Theorem 4.5]{ovu}). In particular, the OT-manifolds with universal cover $\h^{2}\times \C$ give  counterexamples to a conjecture of Vaisman \cite[Page 8]{do} (see also \cite[Section 4.2]{ovu}). Denote the standard coordinates on $\h^{m-1}\times \C$ by $(z_1,\cdots,z_m)$. On OT-manifolds $M_K$ a constant multiple of the product of standard Poincar\'{e} metric $\al=\mn\sum\limits_{i=1}^{m-1}\frac{\md z_i\wedge \md \ov{z_i}}{4(\mathcal{I}\mathrm{m} z_i)^2}$ on $\h^{m-1}$ descends to a closed semipositive real $(1,1)$ form on $M_K$ denoted by $\omega_{\infty}$ (also denoted by $\al$ itself) with
$$
0\leq \omega_{\infty}\in-c_{1}^{\mathrm{BC}}(M_K),
$$
where $c_{1}^{\mathrm{BC}}(M_K)$ is the first Bott-Chern class of $M_K$. The $(1,1)$ form $\omega_{\infty}$ will play a key role in our results.

We consider the normalized Chern-Ricci flow
\be\label{crf}
\ddt\omega=-\ric(\omega)-\omega,\quad\omega|_{t=0}=\omega_{0}
\ee
on $M_K$, with an initial Hermitian metric $\omega_{0}$. Here $\mathrm{Ric}(\omega)$ is the \emph{Chern-Ricci} form of the Hermitian metric $\omega=\mn g_{i\ov{j}}\md z_i\wedge \md\ov{z_j}$ defined by
$$
\mathrm{Ric}(\omega)=-\ddb \log\mathrm{det}g.
$$
Since the canonical bundle of $M_K$ is nef, the results of \cite{twjdg,twcomplexsurface,twymathann} imply that there exists a unique solution to (\ref{crf}) for all time. We are concerned the behavior of the normalized Chern-Ricci flow as $t\longrightarrow \infty$.
\begin{thm}\label{thm1}
Let $M_K$ be an OT-manifold and $\omega$ be any Hermitian metric on $M_K$. Then there exists a Hermitian metric $\omega_{\mathrm{LF}}=e^{\sigma}\omega$ in the conformal class of $\omega$ such that the following holds.

Let $\omega(t)$ be the solution of the normalized Chern-Ricci flow (\ref{crf}) with the initial Hermitian metric of the form
$$
\omega_0=\omega_{\mathrm{LF}}+\ddb \rho>0.
$$
Then as $t\longrightarrow \infty$,
$$
\omega(t)\longrightarrow \omega_{\infty}
$$
uniformly on $M_K$ and exponentially fast, where $\omega_{\infty}$ is the $(1,1)$ form defined above. Furthermore,
$$
(M_K,\omega(t))\longrightarrow (\torus^{m-1},g)
$$
in the Gromov-Haudorff sense, where $g$ defined in (\ref{metrictorus}) is the flat Riemannian metric on torus $\torus^{m-1}$ determined by the OT-manifold $M_K$.
\end{thm}
Therefore, we prove that the (normalized) Chern-Ricci flow collapses a Hermitian metric $\omega$ on the OT-manifold $M_K$ to a torus, modulo an initial conformal change to $\omega$. Indeed, we prove more than this, since our initial Hermitian metric can be any one in the $\de\ov{\de}$-class of $e^{\sigma}\omega$. Note that this collapsing to $\torus^{m-1}$ is in stark contrast with the properties of the K\"{a}hler-Ricci flow which always collapses to even-dimensional manifolds (cf.\cite{fong,gill14,songtian,songtianjams,songtian11,swlecturenotes,twypreprint}).

Our conformal change is relative to a holomorphic foliation structure without singularity $\F$ defined by $\omega_{\infty}$ on the OT-manifold $M_K$. Now we give an outline of the explanation of this holomorphic foliation structure (more details can be found in Section \ref{sectionot}). Note that $\F$ can be induced by the holomorphic foliation $\ti\F$ generated by $\de_{z_m}$ on the universal covering manifold $\h^{m-1}\times \C$ and every leaf of $\ti\F$ is of form
$
\{z'\}\times \C,
$
where $z'\in \h^{m-1}$.
Motivated by \cite{ftwz}, we give some definitions and deduce a useful proposition as follows.
\begin{defn}\label{defn}
A Hermitian metric $\omega$ on $M_K$ is called \emph{flat along the leaves} if the restriction of $\pi^{\ast}\omega$  to every leaf of $\ti\F$ is a flat K\"{a}hler metric on $\C$, and called \emph{strongly flat along the leaves} if this restriction of $\pi^{\ast}\omega$  to every leaf of $\ti\F$ equals to
$$
c((\mathcal{I}\mathrm{m} z_1)\cdots(\mathcal{I}\mathrm{m} z_{m-1}))\mn \md z_{m}\wedge\md\ov{z_m},
$$
where $c>0$ is a constant independent of the leaf.
\end{defn}
The Hermitian metric $\omega_{\mathrm{LF}}$ we need in the statement of Theorem \ref{thm1} is exactly strongly flat along the leaves.
The following proposition shows that the assumption of being strongly flat along the leaves is not  in fact restrictive  because it can always be obtained from any Hermitian metric $\omega$ by a conformal change (see also Lemma \ref{lemleafflat}).
\begin{prop}\label{prop}
For any Hermitian metric $\omega$ on the OT-manifold $M_K$, there exists a smooth function $\sigma\in C^{\infty}(M_K,\,\R)$ such that $\omega_{\mathrm{LF}}:=e^{\sigma}\omega$ is strongly flat along the leaves.
\end{prop}
We remark that in the case of Inoue surfaces $S_M$, Definition \ref{defn} and Proposition \ref{prop} specialize to the corresponding ones in \cite{ftwz}.

Another interesting question is whether we can get the smooth ($C^{\infty}$) convergence of $\omega(t)$ to $\omega_{\infty}$ instead of the uniform ($C^{0}$) convergence in Theorem \ref{thm1}.  In this direction,  if the initial Hermitian metric is of a more restricted type, then we can get $C^{\al}$ convergence for $0<\al<1$. More precisely, Oeljeklaus and Toma \cite{ot} and Ornea and Verbitsky \cite{ov} constructed an explicit Hermitian metric $\omega_{\mathrm{OT}}$ defined in (\ref{otmetric}) on OT-manifold $M_K$ exactly we consider, which is strongly flat along the leaves. For the initial Hermitian metrics in the $\de\ov{\de}$-class of $\omega_{\mathrm{OT}}$ and  $0<\al<1$, we prove the $C^{\al}$ convergence as follows.
\begin{thm}\label{thm3}
Let $\omega(t)$ be the solution of the normalized Chern-Ricci flow (\ref{crf}) on an OT-manifold $M_K$ with an initial Hermitian metric of the form
$$
\omega_0=\omega_{\mathrm{OT}}+\ddb\rho>0.
$$
Then the solution metric $\omega(t)$ are uniformly bounded in the $C^1$ topology and for any $0<\al<1$, there holds
$$
\omega(t)\longrightarrow \omega_{\infty}, \quad \mbox{as}\;\;t\longrightarrow \infty
$$
in the $C^{\al}$ topology.
\end{thm}
We note that while the strategy of the proofs is the same
as in \cite{ftwz}, new difficulties arise due to the fact that these manifolds have
dimension greater than $2$.
This is the first general result
where collapsing of the Chern-Ricci flow in dimensions greater than $2$ is established for a large
class of manifolds and initial metrics.

\noindent {\bf Acknowledgements}
The author thanks Professor Valentino Tosatti and Professor Ben Weinkove for suggesting him the problem and Professor Xiaokui Yang for some helpful conversations. Sincere appreciation goes to Professor Valentino Tosatti for helping the author overcoming the difficulties in the preparation for this paper, pointing out some mistakes in calculation, offering more references and so many other useful comments on an earlier version of this paper.
This work was carried out when the author was visiting the Mathematics Department at Northwestern University and he thanks the department for its hospitality and for providing a good academic environment. The author is also grateful to the anonymous referees and the editor for their careful reading and helpful suggestions which greatly improved the paper.

\section{Oeljeklaus-Toma manifolds}\label{sectionot}
Let $\q$ be the field consisting of rational numbers and $K$ be a finite extension field of $\q$ with degree $[K:\q]=n$.
Then the field $K$ admits precisely $n=s+2t$ distinct embeddings $\sigma_1,\cdots,\,\sigma_n$ into the field $\C$ consisting of complex numbers, where $\sigma_1,\cdots,\,\sigma_s$ are
real embeddings and $\sigma_{s+1},\cdots,\,\sigma_n$ are complex embeddings. Without loss of generation, assume that $\sigma_{s+i}=\ov{\sigma_{s+t+i}}$ for $1\leq i\leq t$ because the complex embeddings of $K$ into $\C$ occur in pairs of complex conjugate embeddings. Also assume that both $s$ and $t$ are positive. Let
$$
\sigma:\;K\longrightarrow \C^m,\quad \sigma(a):=(\sigma_{1}(a),\cdots,\sigma_{s+t}(a))
$$
be the geometric representation of $K$.

Let $O_{K}$ be the ring of algebraic integers of $K$ and $O_{K}^{\ast}$ be the multiplicative group of units of $O_{K}$, that is,
$$
O_{K}^{\ast}:=\{a\in O_K:\;\sigma_1(a)\cdots\sigma_{s}(a)|\sigma_{s+1}(a)|^2\cdots|\sigma_{s+t}(a)|^2=\pm1\}
$$
and also let
$$
O_{K}^{\ast,+}:=\{a\in O_K^{\ast}:\;\sigma_1(a)>0,\cdots,\sigma_{s}(a)>0\}.
$$
It is well-known that the image $\sigma(O_K)$ is a lattice of rank $n$ in $\C^m$, where $m:=s+t$ (see for example \cite[Theorem 1 in Section 3 of Chapter 2]{bs}).
Therefore, we get a properly discontinuous action of $O_K$ on $\C^m$ by translations.

Consider the multiplicative action of $O_{K}$ on $\C^m$ given by
$$
az:=(\sigma_{1}(a)z_1,\cdots,\sigma_{m}(a)z_m).
$$
Denote by $\h$  the upper complex half-plane, that is, $\h=\{z\in \C:\;\mathcal{I}\mathrm{m} z>0\}$.
Since for $a\in O_{K}$, $a\sigma(O_{K})\subset \sigma(O_{K})$, combining the additive action of $O_{K}$ and the multiplicative action of
$O_{K}^{\ast,+}$, Oeljeklaus and Toma \cite{ot} (see also \cite{pv}) obtained a free action of $O_{K}^{\ast,+}\ltimes O_{K}$ on $\h^s\times \C^t$. Now consider the logarithmic representation of units
\begin{align}
& \mathrm{L}:\;O_{K}^{\ast,+}\longrightarrow\R^m,\nonumber\\
&\mathrm{L}(a):=(\log\sigma_{1}(a),\cdots,\log\sigma_{s}(a),2\log|\sigma_{s+1}(a)|,\cdots,2\log|\sigma_{s+t}(a)|). \label{logrep}
\end{align}
It follows from the Dirichlet's Units Theorem (see for example \cite{bs}) that $\mathrm{L}(O_{K}^{\ast,+})$ is a full lattice in the subspace $H$ of $\R^m$, where
$$
H:=\lf\{x\in\R^m:\;\sum\limits_{i=1}^mx_i=0\rt\}.
$$
For $t>0$, the projection $\mathrm{Pr}:\,H\longrightarrow \R^s$ given by the first $s$ coordinate functions is surjective. So there exists subgroups $G$ of rank $s$ of $O_{K}^{\ast,+}$ such that $\mathrm{Pr}\circ \mathrm{L}(G)$ is a full lattice $\Lambda$ in $\R^s$. Such a subgroup is called \emph{admissible} for the field $K$ by Oeljeklaus and Toma \cite{ot}.

Take $G$ admissible for $K$. The quotient manifold $(\h^s\times \C^t)/\sigma(O_K)$ is diffeomorphic to a trivial torus bundle $(\R_{+})^s\times(\s^1)^n$ and $G$ acts properly discontinuously on it because it induces a properly discontinuous action on $(\R_{+})^s$.
Therefore, we get an $m$-dimensional compact complex manifold
$$
M_{K,G}:=(\h^s\times\C^t)/(G\ltimes O_{K})
$$
which is a fiber bundle over $\torus^s:=\underbrace{\s^1\times\cdots\times \s^1}_{s-\mbox{times}}$ with $\torus^n:=\underbrace{\s^1\times\cdots\times \s^1}_{n-\mbox{times}}$ as fiber. Such a manifold is called \emph{Oeljeklaus-Toma (OT-) manifold}.

For $s=t=1,\,G=O_{K}^{\ast,+}$, $M_{K,G}$ is an Inoue-Bombieri surface $S_M$ (see \cite{inoue}).

In this paper, we will consider the OT-manifold in the case of $s>0,\,t=1$ (to prove similar results in the case when $t>1$, it seems
that new ideas will be required (see \cite{ovu,vuletescu})), that is, $M_K:=M_{K,G}=(\h^{m-1}\times \C)/\Gamma$, where $\Gamma:=(G\ltimes O_{K})$ and denote the quotient covering map by $\pi:\,\h^{m-1}\times \C\longrightarrow M_K$ and the fiber projection by
$p:\;M_K\longrightarrow \torus^{m-1}$.

Let
$$
z_i=x_i+\mn y_i,\quad i=1,\cdots m,
$$
where $z_1,\cdots,z_{m-1}$ be the standard coordinates of $\h^{m-1}$ and $z_m$ be the standard coordinate of $\C$.
Then we have some $\Gamma$-invariant forms on $\h^{m-1}\times\C$ which can be induced on $M_K$ and denoted by the same symbols as follows (see \cite{ot,ov}).
$$
\al:=\mn \sum\limits_{i=1}^{m-1}\frac{\md z_i\wedge \md \ov{z_i}}{4y_{i}^2},\;\beta:=\mn (y_1\cdots y_{m-1})\md z_m\wedge\md \ov{z_m},\; \gamma:=\mn\sum\limits_{k,\ell=1}^{m-1}\frac{\md z_{k}\wedge \md \ov{z_{\ell}}}{4y_ky_{\ell}}.
$$
In addition, $\al$ is $\md$-closed and also denoted by $\omega_{\infty}$ when it descends to $M_K$ in Section \ref{section1}.
Therefore, we can construct a Hermitian metric $\omega_{\mathrm{OT}}$ by
\be\label{otmetric}
\omega_{\mathrm{OT}}=\al+\beta+\gamma
\ee
with Ricci form
$$
\ric(\omega_{\mathrm{OT}})=-\al\in c_{1}^{\mathrm{BC}}(M_K).
$$
If we define a function $\psi(z)=(y_1\cdots y_{m-1})^{-1}$ on $\h^{m-1}\times \C$, then $\omega_{\mathrm{OT}}$ was defined in \cite{ot,ov}  to be
$$
\frac{\ddb (\psi(z)+|z_m|^2)}{\psi(z)}
$$
and a simple calculation shows that this equals to $\al+\beta+\gamma$.
In the case of Inoue-Bombieri surfaces,
$
\omega_{\mathrm{T}}=4\al+\beta
$
is called the Tricerri metric \cite{tricerri}.

Now we give more details about the holomorphic foliation $\F$ mentioned in Section \ref{section1}.
We begin with the holomorphic foliation $\ti\F$ without singularity on $\h^{m-1}\times \C$ generated by vector
field $\de_{z_m}$. The foliation $\ti\F$ is $\Gamma$-invariant and is also the kernel of the $\Gamma$-invariant form $\al$.
Therefore, it induces a holomorphic foliation $\F$ without singularity on $M_{K}$ with the kernel $\al=\omega_{\infty}$ (see \cite{ov}). A leaf of the
foliation $\ti\F$ including the point $(t_1,\cdots,t_m)\in \h^{m-1}\times\C$ is given as
$$
\ti{\mathcal{L}}_{t_1,\cdots,t_{m-1}}:=\{(z',\,z_m)\in \h^{m-1}\times\C:\;z'=(t_1,\cdots,t_{m-1})\in \h^{m-1}\}.
$$
Since the isotropy group of the leaf denoted by
$$
G_{\ti{\mathcal{L}}_{t_1,\cdots,t_{m-1}}}:=\{g\in \Gamma:\;g \ti{\mathcal{L}}_{t_1,\cdots,t_{m-1}}=\ti{\mathcal{L}}_{t_1,\cdots,t_{m-1}}\}.
$$
is trivial,
we get a leaf $\mathcal{L}$ of $\F$  via the natural immersion
of
$
\ti{\mathcal{L}}_{t_1,\cdots,t_{m-1}}/G_{\ti{\mathcal{L}}_{t_1,\cdots,t_{m-1}}}
$
diffeomorphic to
$
\ti{\mathcal{L}}_{t_1,\cdots,t_{m-1}}
$
into $M_K$.
All the leaves of $\F$ can be obtained in this way (see \cite{mm}). For  the closure $Z$ of a leaf $\mathcal{L}=\pi(\ti{\mathcal{L}}_{t_1,\cdots,t_{m-1}})$ of $\F$, Ornea and Verbitsky \cite[Proposition 3.2]{ov} proved that
$$
\pi^{-1}(Z)\supseteq Z_{\al_1,\cdots,\al_{m-1}}:=\{(z_1,\cdots,z_m)\in \h^{m-1}\times \C:\;\al_{i}=\mathcal{I}\mathrm{m} z_{i},\;1\leq i\leq m-1\},
$$
where $\al_i=\mathcal{I}\mathrm{m} t_{i},\,1\leq i\leq m-1$ and $Z_{\al_1,\cdots,\al_{m-1}}$ is the closure of $O_K(\ti{\mathcal{L}}_{t_1,\cdots,t_{m-1}})$. Therefore, we can deduce
\begin{lem}
For any point $a\in M_{K}$, the leaf $\mathcal{L}_{a}$ of the foliation $\F$ through this point is
dense in the $\torus^{m+1}$-fiber of the point $p(a)\in \torus^{m-1}$, that is, for any point $t=(t_1,\cdots,t_m)\in \h^{m-1}\times\C$, $\pi(\ti{\mathcal{L}}_{t_1,\cdots,t_{m-1}})$ is dense in the $\torus^{m+1}$-fiber of the point $p\circ\pi(t)\in \torus^{m-1}$.
\end{lem}
The following lemma shows that every Hermitian metric $\omega$ on $M_K$ is conformal to a Hermitian one which is strongly flat along the leaves.
\begin{lem}\label{lemleafflat}
A Hermitian metric $\omega_{\mathrm{LF}}$ on $m$-dimensional $M_{K}$ is flat along the leaves if and only if
\be\label{lemf1}
\al^{m-1}\wedge \omega_{\mathrm{LF}}=(p^{\ast}\eta)\al^{m-1}\wedge \beta,
\ee
where $\eta:\,\mathbb{T}^{m-1}=\underbrace{\s^1\times\cdots\times \s^1}_{(m-1)-\mbox{times}}\longrightarrow \R$ is a smooth positive function. And it is strongly flat along the leaves if and only if
\be\label{lemf2}
\al^{m-1}\wedge\omega_{\mathrm{LF}}=c\al^{m-1}\wedge\beta,
\ee
where $c>0$ is a constant independent of leaf. For any Hermitian metric $\omega$ on $M_K$, define $\sigma \in C^{\infty}(M_K,\,\R)$ by
$$
e^{\sigma}=\frac{\al^{m-1}\wedge \beta}{\al^{m-1}\wedge \omega}.
$$
Then $\omega_{\mathrm{LF}}=e^{\sigma}\omega$ satisfies (\ref{lemf2}) with $c=1$ and hence is strongly flat along the leaves.
\end{lem}
\begin{proof}
Write the pullback of the Hermitian metric $\omega_{\mathrm{LF}}$ as
$$
\pi^{\ast}\omega_{\mathrm{LF}}=\sum\limits_{i,j=1}^{m}g_{i\ov{j}}\md z_i\wedge\md\ov{z_j}
$$
and we have
\be\label{lemp1}
\frac{\al^{m-1}\wedge\pi^{\ast}\omega_{\mathrm{LF}}}{\al^{m-1}\wedge\beta}=\frac{g_{m\ov{m}}}{y_1\cdots y_{m-1}}.
\ee
So (\ref{lemf1}) is equivalent to
$$
\frac{g_{m\ov{m}}}{y_1\cdots y_{m-1}}=\pi^{\ast}p^{\ast}\eta.
$$
Notice that the function $\pi^{\ast}p^{\ast}\eta$ depends only on $(y_1,\cdots,y_{m-1})$. Since the restriction of $\pi^{\ast}\omega_{\mathrm{LF}}$ to a leaf $\{z'\}\times \C$ equals to $\mn g_{m\ov{m}}\md z_{m}\wedge\md \ov{z_{m}}$, and its Ricci curvature equals $-\de_{m}\de_{\ov{m}}\log g_{m\ov{m}}$, we can deduce that if (\ref{lemf1}) holds then $\omega_{\mathrm{LF}}$ is flat along the leaves.

Conversely, if $\omega_{\mathrm{LF}}$ is flat along the leaves, then for each fixed $z'\in \h^{m-1}$ we have that
$$
\de_{m}\de_{\ov{m}}\log\frac{g_{m\ov{m}}}{y_1\cdots y_{m-1}}=\de_{m}\de_{\ov{m}}\log  g_{m\ov{m}}=0.
$$
Thanks to (\ref{lemp1}) we get that the function $\log \frac{g_{m\ov{m}}}{y_1\cdots y_{m-1}}$ on $\h^{m-1}\times \C$ is $\Gamma$-invariant, hence bounded (because it is the pullback of a function from $M_K$). Therefore, $\log \frac{g_{m\ov{m}}}{y_1\cdots y_{m-1}}$ for $(z_1,\cdots,z_{m-1})\in \h^{m-1}$ fixed is a bounded harmonic function on $\C$, and so it must be constant. In other words, the ratio $\frac{\al^{m-1}\wedge\omega_{\mathrm{LF}}}{\al^{m-1}\wedge\beta}$ is constant along each leaf of $\F$. Since every leaf is dense in the $\torus^{m+1}$ fiber which contain it, we obtain that $\frac{\al^{m-1}\wedge\omega_{\mathrm{LF}}}{\al^{m-1}\wedge\beta}$ equals the pullback of a function from $\torus^{m-1}$.

On the other hand, it is now clear that $\omega$ is strongly flat along the leaves if and only if (\ref{lemf2}) holds, or equivalently,
$$
\frac{g_{m\ov{m}}}{y_1\cdots y_{m-1}}=c,
$$
where $c>0$ is a constant. The last assertion of the lemma is immediate.
\end{proof}
To end this section, we give some details about the Riemannian metric on $\torus^{m-1}$ induced from $\al$. On $\h^{m-1}$, $\al$
corresponds to the Riemannian metric
$$
\sum\limits_{i=1}^{m-1}\frac{\md x_i\otimes \md x_i+\md y_i\otimes \md y_i}{2y_{i}^2}
$$
which is restricted   on  $(\R_{+})^{m-1}$
\be\label{metricuse1}
\sum\limits_{i=1}^{m-1}\frac{\md y_i\otimes \md y_i}{2y_{i}^2}.
\ee
Under the local coordinate
$$
f:\;(\R_{+})^{m-1}\longrightarrow \R^{m-1},\quad (y_{1},\cdots,y_{m-1})\longmapsto(\log y_1,\cdots,\log y_{m-1}),
$$
the metric (\ref{metricuse1}) can be expressed as
\be\label{metricuse}
\frac{1}{2}\sum\limits_{i=1}^{m-1}\md x_i\otimes \md x_i.
\ee
Now let $a_1,\cdots,a_{m-1}$ be the generators of the admissible group $G$. Then under the logarithmic representation (\ref{logrep}),
$$
(\log \sigma_{1}(a_i),\cdots,\log\sigma_{m-1}(a_i))=:(v_{i1},\cdots,v_{i,m-1}),\quad i=1,\cdots,m-1
$$
is the basis of the full lattice $\Lambda$ in $\R^{m-1}$ and $\torus^{m-1}=\R^{m-1}/\Lambda$, where $\R^{m-1}$ is equipped with the metric (\ref{metricuse}).
So the metric on $\torus^{m-1}$ is
\be\label{metrictorus}
\frac{1}{2}\sum\limits_{k,\ell=1}^{m-1}\lf(\sum\limits_{i=1}^{m-1}v_{ki} v_{\ell i}\rt)\md x^k\otimes\md x^{\ell},
\ee
and the radius  of the $k$-th ($k=1,\cdots,m-1$) factor $\s^{1}$  of $\torus^{m-1}$ is
$$
\frac{1}{2\sqrt{2}\pi}\lf(\sum\limits_{i=1}^{m-1}v_{ki} v_{k i}\rt)^{1/2}.
$$
Obviously, the metric on $\torus^{m-1}$ depends on the lattice $\Lambda$. In fact, the metrics $g_{\Lambda}$ and $g_{\Lambda'}$ defined on $\torus^{m-1}$ are isometric if and only if there exists an isometry of $\R^{m-1}$ which sends the lattice $\Lambda$ on the lattice $\Lambda'$ (see \cite[Theorem 2.23]{ghl}).
\section{the Chern-Ricci flow on OT-manifolds}
We will write the normalized Chern-Ricci flow as a parabolic complex Monge-Amp\`{e}re equation.
Let $\omega_{\mathrm{LF}}=\mn\sum_{i,j=1}^m(g_{\mathrm{LF}})_{i\ov{j}}\md z_i\wedge \md\ov{z_j}$ be
the Hermitian metric which is strongly flat along the leaves, as in the setup of Theorem \ref{thm1}.
First, we define
\be\label{refmetric}
\ti\omega=\ti\omega(t)=e^{-t}\omega_{\mathrm{LF}}+(1-e^{-t})\al>0,
\ee
and denote by $\ti g$ the Hermitian metric associated to $\ti \omega$. We define a volume form $\Omega$ by
\be\label{vol}
\Omega=m\al^{m-1}\wedge\omega_{\mathrm{LF}}=mc\al^{m-1}\wedge\beta,
\ee
with the constant $c$ defined by (\ref{lemf2}). Direct calculation using (\ref{vol}) implies
$$
\ddb \log\Omega=\al.
$$
It follows that the normalized Chern-Ricci flow (\ref{crf}) is equivalent to the parabolic complex Monge-Amp\`{e}re equation
\be\label{cma}
\ddt \varphi=\log \frac{e^t(\ti\omega+\ddb\varphi)^m}{\Omega}-\varphi,\quad \ti\omega+\ddb\varphi>0,\quad\varphi(0)=\rho.
\ee
Namely, if $\varphi$ solves the equation (\ref{cma}), then $\omega(t)=\ti\omega+\ddb\varphi$ solves
the normalized Chern-Ricci flow (\ref{crf}), as is readily checked. Conversely, given a
solution $\omega(t)$ of (\ref{crf}) we can find a solution (see \cite{twjdg}) $\varphi=\varphi(t)$ of (\ref{cma})
such that $\omega(t)=\ti\omega+\ddb \varphi$.

Let $\varphi=\varphi(t)$ be the solution to (\ref{cma}) and write $\omega=\omega(t)=\ti\omega+\ddb\varphi$ for
the corresponding Hermitian metrics along the normalized Chern-Ricci flow (\ref{crf}).
We first prove uniform estimates on the potential $\varphi$ and its time derivative $\dot{\varphi}$.
Given the choice of $\ti\omega$ and $\Omega$, the proof is very similar to the one
in \cite[Lemma 3.6.3, Lemma 3.6.7]{swlecturenotes} (see also \cite{ftwz,fong,gill14,songtian,twymathann}).
\begin{lem}\label{ordinarylemma}
There exists a uniform positive constant $C$ such that on $M_K\times [0,\,\infty)$
\begin{enumerate}
  \item[(\romannumeral1)] $|\varphi|\leq C(1+t)e^{-t}$.
  \item[(\romannumeral2)] $|\dot{\varphi}|\leq C$.
  \item[(\romannumeral3)] $C^{-1}\ti\omega^m\leq \omega^m\leq C\ti\omega^m$.
\end{enumerate}
\end{lem}
\begin{proof}
Since the discussion are very similar to those in \cite{gill14,swlecturenotes,twymathann}, we will be brief.
For part (\romannumeral1), first we claim that, by the choice of $\ti\omega$ and $\Omega$, there holds
\be\label{varphibound}
\lf|e^t\log\frac{e^{t}\ti\omega^m}{\Omega}\rt|\leq C',
\ee
for uniform $C'$. Indeed,  from (\ref{refmetric}) and (\ref{vol}), we have
\begin{align}
\frac{e^{t}\ti\omega^m}{\Omega}
=&\frac{e^{t}\sum\limits_{k=0}^m
\binom{m}{k}(1-e^{-t})^k\al^ke^{-(m-k)t}\omega_{\mathrm{LF}}^{m-k}}{\Omega}\label{volequa}\\
=&\frac{m(1-e^{-t})^{m-1}\al^{m-1}\wedge\omega_{\mathrm{LF}}+e^{-t}\sum\limits_{k=0}^{m-2}
\binom{m}{k}(1-e^{-t})^ke^{-(m-2-k)t}\al^k\wedge\omega_{\mathrm{LF}}^{m-k}}{\Omega}\nonumber\\
=&1+O(e^{-t})\nonumber,
\end{align}
which implies (\ref{varphibound}). From now on, $O(f(t))$ will mean $\leq C f(t)$ for a uniform constant $C$, where $f(t)$ is a positive function of $t$ (e.g., $e^{-t},\,1,\,e^t$). Now consider the quantity
$$
P=e^t\varphi-(C'+1)t.
$$
If $\sup_{M_K\times[0,\,t_0]}P=P(x_0,\,t_0)$ for some $x_0\in M_K$ and $t_0>0$, we have at this point,
$$
0\leq \frac{\de P}{\de t}\leq e^t\log\frac{e^{t}\ti\omega^m}{\Omega}-C'-1\leq -1,
$$
absurd. Therefore, $\sup_{M_K}P$ is bounded from above by its initial value, which implies $\varphi\leq C(1+t)e^{-t}$. The lower bound is similar.

To prove (\romannumeral2), choose a constant $C_0$ satisfying $C_{0}\ti\omega>\al$ for all $t\geq 0$. Then compute, for the Laplacian $\Delta=g^{\ov{j}i}\de_{i}\de_{\ov{j}}$,
\begin{align*}
\lf(\ddt-\Delta\rt)(\dot{\varphi}-(C_0-1)\varphi)
=&1+\tr_{\omega}(\al-\ti\omega)-C_{0}\dot{\varphi}+(C_{0}-1)\tr_{\omega}(\omega-\ti\omega)\\
<&1-C_{0}\dot{\varphi}+m(C_{0}-1).
\end{align*}
The maximum principle implies that $\dot{\varphi}$ is bounded from above. For the lower bound of $\dot{\varphi}$,
\be\label{dotvarphilower}
\ba
\lf(\ddt-\Delta\rt)(\dot{\varphi}+2\varphi)
=&\tr_{\omega}(\al-\ti\omega)+1+\dot{\varphi}-2\tr_{\omega}(\omega-\ti\omega)\\
\geq&\tr_{\omega}\ti\omega+\dot{\varphi}-(2m-1).
\ea
\ee
By the geometric-arithmetic means inequality, we have
\be\label{agm}
e^{-\frac{\dot{\varphi}+\varphi}{m}}=\lf(\frac{\Omega}{e^t\omega^m}\rt)^{\frac{1}{m}}
\leq C\lf(\frac{\ti\omega^m}{\omega^m}\rt)^{\frac{1}{m}}\leq \frac{C}{m}\tr_{\omega}\ti\omega,
\ee
where we use (\ref{volequa}).
Combining (\ref{dotvarphilower}), (\ref{agm}) and the maximum principle indicates that $\dot{\varphi}$ is bounded from below.

Finally, (\romannumeral3) follows from (\romannumeral1), (\romannumeral2) and the equation (\ref{cma}).
\end{proof}
Next, we bound the torsion and curvature of the reference metrics $\ti g$. We will denote the Chern connection, torsion and curvature of $\ti g$
by $\ti\n,\,\ti T$ and $\widetilde{\mathrm{Rm}}$ respectively, and also write
$$
\ti T_{ij\ov{\ell}}=\ti T_{ij}^{k}\ti g_{k\ov{\ell}}=\de_{i}\ti g_{j\ov{\ell}}-\de_{j}\ti g_{i\ov{\ell}}.
$$
Since $\al$ is a closed form, we have
\be\label{Tti}
\ti T_{ij\ov{\ell}}=e^{-t}(T_{\mathrm{LF}})_{ij\ov{\ell}},
\ee
where $T_{\mathrm{LF}}$ is  the torsion of the metric $g_{\mathrm{LF}}$.  We can deduce the following bounds on torsion and curvature of $\ti g$, which are analogous to those in \cite[Lemma 4.1]{twymathann}.
\begin{lem}\label{keylemma}
There exists a uniform constant $C$ such that
\begin{enumerate}
  \item[(\romannumeral1)]$|\ti T|_{\ti g}\leq C$.
  \item[(\romannumeral2)]$|\ov{\de}\ti T|_{\ti g}+|\ti \n\ti T|_{\ti g}+|\widetilde{\mathrm{Rm}}|_{\ti g}\leq Ce^{t/2}$.
\end{enumerate}
\end{lem}
\begin{proof}
Denote by $\ti g_{k\ov{\ell}}$ the component of metric matrix $(\tilde{g}_{p\ov{q}})$ in $k$-th row and $\ell$-th column.
We have
\begin{align}
\ti g_{k\ov{k}}
=&e^{-t}(g_{\mathrm{LF}})_{k\ov{k}}+(1-e^{-t})\al_{k\ov{k}}=O(1),\quad 1\leq k\leq m-1,\no\\
\ti g_{k\ov{\ell}}
=&e^{-t}(g_{\mathrm{LF}})_{k\ov{\ell}}=O(e^{-t}),\quad \mbox{otherwise}.\nonumber
\end{align}
Denote by $\ti g^{\ov{\ell}k}$ the component of the inverse matrix of $(\ti g_{p\ov{q}})$ in $\ell$-th row and $k$-th column, and by $G_{k\ell}$ the algebraic cofactor of the component $\ti g_{k\ov{\ell}}$.
Note that
\begin{align*}
G_{mm}=&\ti g_{1\ov{1}}\cdots\ti g_{m-1\ov{m-1}}+\sum\mbox{terms with factor}\;e^{-t} \\
=&\al_{1\ov{1}}\cdots\al_{m-1\ov{m-1}}+\sum\mbox{terms with factor}\;e^{-t}
\end{align*}
and
$$
G_{k\ell}=\sum \mbox{terms with factor}\; e^{-t},\quad (k,\ell)\neq (m,m).
$$
A preliminary analysis implies that there exists a uniform constant $c_0$ independent of $t$ such that
\begin{align*}
\mathrm{det}(\ti g_{p\ov{q}})
=\sum\limits_{\ell=1}^m\ti g_{m\ov{\ell}}G_{m\ell}
=e^{-t}\sum\limits_{\ell=1}^m(g_{\mathrm{LF}})_{m\ov{\ell}}G_{m\ell}\geq c_0 e^{-t},
\end{align*}
where we also use the fact that $(g_{\mathrm{LF}})_{m\ov{m}}>0$.
Therefore, we can deduce that all the components of the inverse metric matrix  $(\ti g^{\ov{q}p})$ are bounded by $O(1)$ except that $\ti g^{\ov{m}m}
$ is bounded by $O(e^{t})$, where we use the formula $\ti g^{\ov{\ell}k}=\frac{G_{k\ell }}{\mathrm{det}(\ti g_{p\ov{q}})}$.

From (\ref{Tti}), we have
\be
|\ti T|_{\ti g}^2=e^{-2t}(T_{\mathrm{LF}})_{ik\ov{q}}\ov{(T_{\mathrm{LF}})_{j\ell \ov{p}}}\ti g^{\ov{j}i}\ti g^{\ov{\ell}k}\ti g^{\ov{q}p}\leq C,
\ee
since the only term involving the cube of $\ti g^{\ov{m}m}$ vanishes by the skew-symmetry of $(T_{\mathrm{LF}})_{ik\ov{q}}$ in $i$ and $k$, and by
the bounds of other components of $(\ti g^{\ov{\ell}k})$ all other terms are bounded.

Since if one of the indexes $1\leq i,\;k,\;\ell\leq m$ equals to $m$, we have
\be\label{use1}
\de_{i}\ti g_{k\ov{\ell}}=\de_{\ov{i}}\ti g_{k\ov{\ell}}=O(e^{-t}),
\ee
and
\be\label{use2}
\de_{m}\ti g_{m\ov{m}}=\de_{\ov{m}}\ti g_{m\ov{m}}=0,
\ee
where for (\ref{use2}) we uses the fact that $\ti g_{m\ov{m}}=ce^{-t}\lf(y_1\cdots y_{m-1}\rt)$,
we can bound on the $\ti g$ norm of the Christoffel symbols $\ti\Gamma_{ik}^p$ of the Chern connection of $\ti g$ by
\be\label{reviseduse1}
|\ti \Gamma_{ik}^p|_{\ti g}^2=\ti \Gamma_{ik}^p\ov{\ti \Gamma_{j\ell}^q}\ti g^{\ov{j}i}\ti g^{\ov{\ell}k}\ti g_{p\ov{q}}
=\ti g^{\ov{j}i}\ti g^{\ov{\ell}k}\ti g^{\ov{q}p}\de_{i}\ti g_{k\ov{q}}\de_{\ov{j}}\ti g_{p\ov{\ell}}\leq C.
\ee
Note that the quantity $|\ti \Gamma_{ik}^p|_{\ti g}^2$ is only locally defined.

Since
$$
(T_{\mathrm{LF}})_{im\ov{m}}=\de_{i}(g_{\mathrm{LF}})_{m\ov{m}}-\de_{m}(g_{\mathrm{LF}})_{i\ov{m}}=O(1),
$$
and
$$
\de_{\ov{m}}(T_{\mathrm{LF}})_{im\ov{m}}
=\de_{i}\de_{\ov{m}}(g_{\mathrm{LF}})_{m\ov{m}}-\de_{m}\de_{\ov{m}}(g_{\mathrm{LF}})_{i\ov{m}}
=-\de_{m}\de_{\ov{m}}(g_{\mathrm{LF}})_{i\ov{m}}
=O(1),
$$
we have, using the skew-symmetry of $(T_{\mathrm{LF}})_{ik\ov{q}}$ in $i$ and $k$,
$$
|(T_{\mathrm{LF}})_{ij\ov{r}}|_{\ti g}^2\leq Ce^{2t},\quad|\de_{\ov{\ell}}(T_{\mathrm{LF}})_{ij\ov{k}}|_{\ti g}^2\leq Ce^{3t}.
$$
Therefore,  from (\ref{Tti}) and the Cauchy-Schwarz inequality, we have
\begin{align}
|\dbar\ti T|_{\ti g}^2
=&|\ov{\n}\ti T|_{\ti g}^2
=e^{-2t}|\de_{\ov{\ell}}(T_{\mathrm{LF}})_{ij\ov{k}}-\ov{\ti\Gamma_{\ell k}^r}(T_{\mathrm{LF}})_{ij\ov{r}}|_{\ti g}^2\nonumber\\
\leq&2e^{-2t}|\de_{\ov{\ell}}(T_{\mathrm{LF}})_{ij\ov{k}}|_{\ti g}^2+2e^{-2t}|\ov{\ti\Gamma_{\ell k}^r}|_{\ti g}^2|(T_{\mathrm{LF}})_{ij\ov{r}}|_{\ti g}^2\nonumber\\
\leq&2e^{-2t}|\de_{\ov{\ell}}(T_{\mathrm{LF}})_{ij\ov{k}}|_{\ti g}^2+Ce^{-2t}|(T_{\mathrm{LF}})_{ij\ov{r}}|_{\ti g}^2\leq Ce^t.\nonumber
\end{align}
Similarly, we can deduce
$
|\n\ti T|_{\ti g}^2\leq Ce^{t}.
$

Recall that the curvature of the Chern connection of $\ti g$ is given by
$$
\ti R_{i\ov{j}k\ov{\ell}}=-\de_{i}\de_{\ov{j}}\ti g_{k\ov{\ell}}+\ti g^{\ov{q}p}\de_{i}\ti g_{k\ov{q}}\de_{\ov{j}}\ti g_{p\ov{\ell}}.
$$
For the bound of $|\ti R_{i\ov{j}k\ov{\ell}}|_{\ti g}^2$,  using (\ref{reviseduse1}), we just need to bound
$$
|\de_{i}\de_{\ov{j}}\ti g_{k\ov{\ell}}|_{\ti g}^2.
$$
Thanks to (\ref{use1}) and (\ref{use2}), we can obtain
\begin{align*}
|\de_{i}\de_{\ov{j}}\ti g_{k\ov{\ell}}|_{\ti g}^2
=& \sum\limits_{k,\ell=1}^{m-1}
\lf(\de_{m}\de_{\ov{m}}\ti g_{k\ov{m}}\rt)\lf(\de_{m}\de_{\ov{m}}\ti g_{m\ov{\ell}}\rt)\ti g^{\ov{m}m}\ti g^{\ov{m}m}\ti g^{\ov{m}m}\ti g^{\ov{\ell}k}\\
&+\mbox{terms bounded by constant}\leq Ce^t.
\end{align*}
So
$
|\widetilde{\mathrm{Rm}}|_{\ti g}^2\leq Ce^{t},
$
as required.
\end{proof}
Now we can apply the arguments of \cite{twymathann,twypreprint} to establish the estimates of the solution $\omega(t)$ to the normalized Chern-Ricci flow (\ref{crf}) and also the solution $\varphi(t)$ to the parabolic complex Monge-Amp\`{e}re equation (\ref{cma}).
\begin{thm}\label{thmuse}
For $\varphi=\varphi(t)$ solving (\ref{cma}) on $M_K$, there holds the following estimates.
\begin{enumerate}
\item[(\romannumeral1)] There exists a uniform constant $C$ such that
$$
C^{-1}\ti\omega\leq \omega(t)\leq C\ti\omega.
$$
\item[(\romannumeral2)] The Chern scalar curvature $R$ satisfies the bound
$$
-C\leq R\leq Ce^{t/2},
$$
where $C$ is uniform constant.
\item[(\romannumeral3)]For any $\eta\in (0,\,1/2)$ and $\sigma\in(0,\,1/4)$, there exists a constant $C_{\eta,\sigma}$ such that
$$
-C_{\eta,\sigma}e^{-\eta t}\leq \dot{\varphi}\leq C_{\eta,\sigma}e^{-\sigma t}.
$$
Therefore, combining part (\romannumeral1) in Lemma \ref{ordinarylemma} and this bound gives (taking $\eta=\sigma$)
$$
|\varphi+\dot{\varphi}|\leq C_{\sigma}e^{-\sigma t}.
$$
\item[(\romannumeral4)]For any $\varepsilon\in(0,\,1/8)$, there exists a constant $C_{\varepsilon}$ satisfying
$$
\|\omega-\ti\omega\|_{C^{0}(M_K,\,\omega_{0})}\leq C_{\varepsilon}e^{-\varepsilon t}.
$$
\end{enumerate}
\end{thm}
\begin{proof}
Given Lemma \ref{ordinarylemma} and Lemma \ref{keylemma}, the proof is almost identical to the discussion in \cite{twymathann,twypreprint}. Therefore, we give only a brief outline and point out the main differences.

For part (\romannumeral1), we claim that
\be\label{x}
\lf(\ddt-\Delta\rt)\log \tr_{\ti\omega}\omega\leq \frac{2}{(\tr_{\ti\omega}\omega)^2}\mathrm{Re}\lf(\ti g^{\ov{\ell}i}g^{\ov{q}k}
\ti T_{ki\ov{\ell}}\de_{\ov{q}}\tr_{\ti\omega}\omega\rt)+Ce^{t/2}\tr_{\omega}\ti\omega,\quad t\geq0.
\ee
Indeed, this inequality can be obtained by the argument which is almost identical to the one in \cite[Lemma 5.2]{twymathann}.
From \cite[Proposition 3.1]{twjdg} we have
\be\label{equlog}
\left( \ddt{} - \Delta \right) \log \mathrm{tr}_{\ti{\omega}} \omega =(I) + (II) + (III)- \frac{1}{\tr_{\ti{\omega}}\omega}\ti{g}^{\ov{\ell}p}\ti{g}^{\ov{q}k}\al_{p\ov{q}}g_{k\ov{\ell}}
\ee
where
$$
\ba
(I) = & \frac{1}{\tr_{\ti{\omega}}\omega} \bigg[-g^{\ov{j}p}g^{\ov{q}i}\ti{g}^{\ov{\ell}k}\ti{\nabla}_k g_{i\ov{j}}\ti{\nabla}_{\ov{\ell}}g_{p\ov{q}} +\frac{1}{\tr_{\ti{\omega}}\omega} g^{\ov{\ell}k}\de_k \lf(\tr_{\ti{\omega}}{\omega}\rt)\de_{\ov{\ell}}\lf(\tr_{\ti{\omega}}{\omega}\rt)\\
&\quad\quad\;\;
- 2\mathrm{Re}\left(g^{\ov{j}i}\ti{g}^{\ov{\ell}k} \ti{T}^p_{ki}\ti{\nabla}_{\ov{\ell}}g_{p\ov{j}}\right) -g^{\ov{j}i}\ti{g}^{\ov{\ell}k}\ti{T}^p_{ik}\ov{\ti{T}^q_{j\ell}}g_{p\ov{q}}\bigg] \\
(II) =& \frac{1}{\tr_{\ti{\omega}}\omega} \bigg[ g^{\ov{j}i}\ti{g}^{\ov{\ell}k}\lf(\ti{\nabla}_i\ov{\ti{T}^q_{j\ell}}-\ti{R}_{i\ov{\ell}p\ov{j}}\ti{g}^{\ov{q}p}\rt)g_{k\ov{q}}
\bigg] \\
(III) =& - \frac{1}{\tr_{\ti{\omega}}\omega} \bigg[  g^{\ov{j}i}\left(  \ti{\nabla}_i\ov{\ti{T}^{\ell}_{j\ell}}   \right) +
 \left(\ti{\nabla}_{\ov{\ell}} \ti{T}^p_{ik}\right)g^{\ov{j}i}\ti{g}^{\ov{\ell}k} \ti{g}_{p \ov{j}}- \ov{\ti{T}^q_{j\ell}} \ti{T}^p_{ik}g^{\ov{j}i}\ti{g}^{\ov{\ell}k} \ti{g}_{p \ov{q}}\bigg].
\ea
$$
Let us point out some differences from the calculation in \cite{twjdg}. Here $\omega$ is evolved by normalized Chern-Ricci flow (\ref{crf}) and our reference metric $\ti\omega$ also depends on time. In particular, in our case we have $T_{ij\ov{\ell}}=\ti T_{ij\ov{\ell}}$ (instead of $T_{ij\ov{\ell}}=(T_{0})_{ij\ov{\ell}}$ in \cite{twjdg}) and the metric $\hat{g}$ and $g_{0}$ in \cite{twjdg} are replaced by $\ti\omega$. The last term in $(\ref{equlog})$ comes from the $-\omega$ term on the right side of (\ref{crf}) and the time derivative of $\ti\omega$. Fortunately,  we have
$$
- \frac{1}{\tr_{\ti{\omega}}\omega}\ti{g}^{\ov{\ell}p}\ti{g}^{\ov{q}k}\al_{p\ov{q}}g_{k\ov{\ell}}\leq 0.
$$
Proposition 3.1 in \cite{twjdg} gives us
$$
(I)\leq  \frac{2}{\lf(\tr_{\ti{\omega}}\omega\rt)^2}\mathrm{Re}\left(
\ti{g}^{\ov{\ell}i}g^{\ov{q}k}  \ti{T}_{k i\ov{\ell}}  \de_{\ov{q}}\lf(\tr_{\ti{\omega}}\omega\rt)
  \right).
$$
Therefore, to complete the proof of the lemma, we only need to show
$$
(II)+(III)\leq Ce^{t/2}\tr_{\omega}\ti\omega.
$$
To see this, from (\romannumeral3) of Lemma \ref{ordinarylemma} and the geometric-arithmetic means inequality, we can deduce $\tr_{\ti\omega}\omega\geq C^{-1}$ for a uniform constant $C$, which, combining with $|g|_{\ti g}\leq \tr_{\ti\omega}\omega$ and
$|g^{-1}|_{\ti g}\leq \tr_{\omega}\ti\omega$, implies
\begin{align}
\frac{1}{\tr_{\ti{\omega}}{\omega}} \lf| g^{\ov{j}i}\ti{g}^{\ov{\ell}k}\ti{\nabla}_i\ov{\ti{T}^q_{j\ell}} g_{k\ov{q}} \rt|
\leq& \frac{1}{\tr_{\ti{\omega}}{\omega}} \lf| g^{-1} \rt|_{\ti{g}} \lf| \ti{g}^{-1}\rt|_{\ti{g}} | \ov{\ti{\nabla}} \ti{T}|_{\ti{g}} |g|_{\ti{g}} \leq  C (\tr_{\omega}{\ti{\omega}})e^{t/2},\nonumber\\
\frac{1}{\tr_{\ti{\omega}}{\omega}} \lf|g^{\ov{j}i}\ti{g}^{\ov{\ell}k}\ti{g}^{\ov{q}p}g_{k\ov{q}}\ti{R}_{i\ov{\ell}p\ov{j}}\rt|
\leq& \frac{1}{\tr_{\ti{\omega}}{\omega}} \lf| g^{-1} \rt|_{\ti{g}}\lf| \ti{g}^{-1}\rt|^2_{\ti{g}} |g|_{\ti{g}} | \widetilde{\textrm{Rm}}|_{\ti{g}} \leq  C(\tr_{\omega}{\ti{\omega}})e^{t/2},\nonumber\\
\frac{1}{\tr_{\ti{\omega}}{\omega}} \lf|g^{\ov{j}i}\ti{\nabla}_i   \ov{\ti{T}^\ell_{j\ell}}\rt |
\leq& \frac{1}{\tr_{\ti{\omega}}{\omega}} \lf| g^{-1} \rt|_{\ti{g}} | \ov{\ti{\nabla}} \ti{T}|_{\ti{g}}\leq C(\tr_{\omega}{\ti{\omega}}) e^{t/2},\nonumber\\
\frac{1}{\tr_{\ti{\omega}}{\omega}}\lf | g^{\ov{j}i}\ti{g}^{\ov{\ell}k}\ti{g}_{p \ov{j}}\ti{\nabla}_{\ov{\ell}}  \ti{T}^p_{ik}  \rt|
\leq&
\frac{1}{\tr_{\ti{\omega}}{\omega}} \lf| g^{-1} \rt|_{\ti{g}} \lf|\ti{g}^{-1}\rt|_{\ti{g}} \lf|\ti{g}\rt|_{\ti{g}} |\ov{\ti{\nabla}} \ti{T}|_{\ti{g}}\leq C(\tr_{\omega}{\ti{\omega}})e^{t/2},\nonumber\\
\frac{1}{\tr_{\ti{\omega}}{\omega}} \lf| g^{\ov{j}i}\ti{g}^{\ov{\ell}k} \ti{T}^p_{ik}\ov{\ti{T}^q_{j\ell}} \ti{g}_{p \ov{q}}\rt|
\leq &\frac{1}{\tr_{\ti{\omega}}{\omega}}
\lf| g^{-1} \rt|_{\ti{g}} \lf|\ti{g}^{-1}\rt|_{\ti{g}} \lf|\ti{g}\rt|_{\ti{g}} |\ti{T}|^2_{\ti{g}}\leq C(\tr_{\omega}{\ti{\omega}}),\nonumber
\end{align}
which completes the proof of the claim.

To prove part (\romannumeral1), first note that part (\romannumeral1) of Lemma \ref{ordinarylemma} implies that $e^{t/2}\varphi$ is uniformly bound. Using the idea from Phong-Sturm {\cite{phongsturm}}, we consider the quantity
$$
Q= \log \tr_{\ti{\omega} }{\omega} - A e^{t/2} \varphi + \frac{1}{e^{t/2}\varphi+\tilde{C} },
$$
where $\ti C$ is a constant satisfying $e^{t/2}\varphi+\tilde{C}\geq 1$ and $A$ is a large constant which will be determined later. Notice that
$$
0\leq \frac{1}{e^{t/2}\varphi+\tilde{C} }\leq 1.
$$
Since $\Delta \varphi=m-\tr_{\omega}\ti\omega$, using the bounds for $\varphi$ and $\dot{\varphi}$ in Lemma \ref{ordinarylemma}, we have
\begin{align}
\left( \ddt{} - \Delta \right) &\left( - A e^{t/2}  \varphi + \frac{1}{e^{t/2}\varphi+\tilde{C}}\right)\nonumber   \\
 =&- \left( A + \frac{1}{\lf(e^{t/2}\varphi+\tilde{C}\rt)^2} \right)  \lf(e^{t/2}\dot{\varphi}+\frac{1}{2}e^{t/2}\varphi \rt) \nonumber \\
 &+ \left( A + \frac{1}{\lf(e^{t/2}\varphi+\tilde{C}\rt)^2} \right) \Delta \lf(e^{t/2} \varphi\rt) - \frac{2 \lf| \partial \lf(e^{t/2}  \varphi \rt)\rt |_{g}^2}{\lf(e^{t/2}\varphi+\tilde{C}\rt)^3} \nonumber \\
 \le & CAe^{t/2}  - A e^{t/2} \tr_{\omega}{\ti{\omega}} - \frac{2 \lf| \partial\lf(e^{t/2}   \varphi \rt) \rt|_{g}^2}{\lf(e^{t/2}\varphi+\tilde{C}\rt)^3}.\label{eqnet2}
\end{align}
At the point $(x_0,\,t_0)$ with $t_{0}>0$ where $Q$ attains a maximum, we have $\de_{\ov{q}}Q=0$, implying
$$
\frac{\partial_{\ov{q}} \tr_{\ti{\omega}}{\omega}}{\tr_{\ti{\omega}}{\omega}} = \left( A + \frac{1}{\lf(e^{t_0/2}\varphi+\tilde{C}\rt)^2} \right)   e^{t_0/2}\partial_{\ov{q}} \varphi.
$$
Then at this point,
\begin{align}
&\frac{2}{(\tr_{\ti{\omega}}{\omega})^2} \textrm{Re} \left( \ti{g}^{\ov{\ell}i} g^{\ov{q}k} \ti{T}_{k i}^p\ti{g}_{p \ov{\ell}} \partial_{\ov{q}} \tr_{\ti{\omega}}{\omega} \right)\nonumber \\
=& \frac{2}{\tr_{\ti{\omega}}{\omega}} \textrm{Re} \left( \ti{g}^{\ov{\ell}i} g^{\ov{q}k} \ti{T}_{k i}^p\ti{g}_{p \ov{\ell}} \left( A + \frac{1}{\lf(e^{t_0/2}\varphi+\tilde{C}\rt)^2} \right)    e^{t_0/2} \partial_{\ov{q}} \varphi \right)\nonumber \\
\leq& \frac{CA^2}{(\tr_{\ti{\omega}}{\omega})^2} \lf(e^{t_0/2}\varphi+\tilde{C}\rt)^3  g^{\ov{q}k} \ti{T}_{ki}^i\overline{\ti{T}_{qr}^r} + \frac{\lf | \partial   \lf(e^{t_0/2}\varphi\rt)\rt |^2_g}{\lf(e^{t_0/2}\varphi+\tilde{C}\rt)^3} \nonumber\\
\le& \frac{CA^2}{\lf(\tr_{\ti{\omega}}{\omega}\rt)^2} \tr_{\omega}\ti\omega  + \frac{ \lf| \partial   \lf(e^{t_0/2}\varphi\rt)\rt |^2_g}{\lf(e^{t_0/2}\varphi+\tilde{C}\rt)^3},\label{dagger}
\end{align}
where for the last step we used Lemma \ref{ordinarylemma}, Lemma \ref{keylemma} and the Cauchy-Schwarz inequality for
the quantity
$
g^{\ov{q}k} \ti{T}_{ki}^i\overline{\ti{T}_{qr}^r}.
$
Combining \eqref{x}, \eqref{eqnet2} and \eqref{dagger}, we have, at a point $(x_0,\,t_0)$, for a uniform $C>0$,
\begin{align*}
\left( \ddt{} - \Delta \right) Q & \le  CA^2\tr_{\omega}{\ti{\omega}}  + C e^{t_0/2}  \tr_{\omega}{\ti{\omega}} + CA  e^{t_0/2}- A e^{t_0/2}  \tr_{\omega}{\ti{\omega}}\\
\leq&CA^2\tr_{\omega}{\ti{\omega}}+ C e^{t_0/2}  \tr_{\omega}{\ti{\omega}} + CA  e^{t_0/2}- A e^{t_0/2}  \tr_{\omega}{\ti{\omega}}
\end{align*}
where we are assuming, without loss of generality, that at this maximum point of $Q$ we have $\tr_{\ti\omega}\omega\geq 1$. Choose a uniform $A$ satisfying $A\geq C+1$. We can also assume $t_0>T_0$, where $T_0$ satisfies
$$
CA^2-e^{t/2}\leq -1,\quad \forall\; t\geq T_0.
$$
Then we can deduce at the maximum of $Q$
$$
\lf[\tr_{\omega}{\ti{\omega}}\rt](x_0,\,t_0)\leq \frac{CA  e^{t_0/2}}{e^{t_0/2}-CA^2}\in (CA,\,CA+C^2A^3),
$$
implying that $\lf[\tr_{\ti\omega}{\omega}\rt](x_0,\,t_0)$ is bounded from above, by using
part (\romannumeral3) of Lemma \ref{ordinarylemma} and the fact (see for example \cite[Corollary 3.5]{swlecturenotes})
\be\label{metricfact}
\tr_{\ti\omega}{{\omega}}
\leq \frac{1}{(m-1)!}\frac{{\omega}^m}{\ti\omega^m}\lf(\tr_{{\omega}}{\ti\omega}\rt)^{m-1}.
\ee
Combining the upper bound of $\lf[\tr_{\ti\omega}{\omega}\rt](x_0,\,t_0)$, the definition of $Q$ and part (\romannumeral1) of Lemma \ref{ordinarylemma} gives the uniform estimate
\be\label{uniformestiamte}
\tr_{\ti{\omega}}{\omega}\leq C
\ee
Again using (\ref{metricfact}) and (\romannumeral3) of Lemma \ref{ordinarylemma}, the uniform upper bound (\ref{uniformestiamte}) shows that $\omega$ is equivalent to $\ti\omega$.

As to part (\romannumeral2), we claim that there exists a uniform constant $C>0$ satisfying
\begin{align}
\left( \ddt{} - \Delta \right) \tr_{\ti{\omega}}{\omega}\leq& -C^{-1}|\ti{\nabla}g|^2_g+Ce^{t/2} \label{nice1}\\
\left( \ddt{} - \Delta \right) \tr_{\omega}{\al}\leq& |\ti{\nabla}g|^2_g-C^{-1}|\nabla \tr_{\omega}{\al}|^2_g+Ce^{t/2} \label{nice2}
\end{align}
along the normalized Chern-Ricci flow (\ref{crf}). As a result, there exists uniform positive constants $C_0$ and $C_1$ such that for all $t\geq 0$
\begin{equation}\label{nice3}
\left( \ddt{} - \Delta \right) (\tr_{\omega}{\al}+C_0\tr_{\ti{\omega}}{\omega})\leq -|\ti{\nabla}g|^2_g-C_1^{-1}|\nabla \tr_{\omega}{\al}|^2_g+C_1e^{t/2}.
\end{equation}
Indeed, the claim follows from the argument similar to the one in \cite[Lemma 6.2]{twymathann} (see also \cite{ftwz}).
From \cite[Proposition 3.1]{twjdg}, as the argument in the proof of inequality (\ref{equlog}), we have
$$
\left( \ddt{} - \Delta \right) \mathrm{tr}_{\ti{\omega}} \omega =J_1 + J_2 + J_3-\ti{g}^{\ov{\ell}p}\ti{g}^{\ov{q}k}\al_{p\ov{q}}g_{k\ov{\ell}}
$$
where
$$
\ba
J_1=&-g^{\ov{j}p}g^{\ov{q}i}\ti{g}^{\ov{\ell}k}\ti{\nabla}_k g_{i\ov{j}}\ti{\nabla}_{\ov{\ell}}g_{p\ov{q}}
- 2\mathrm{Re}\left(g^{\ov{j}i}\ti{g}^{\ov{\ell}k} \ti{T}^p_{ki}\ti{\nabla}_{\ov{\ell}}g_{p\ov{j}}\right)
 -g^{\ov{j}i}\ti{g}^{\ov{\ell}k}\ti{T}^p_{ik}\ov{\ti{T}^q_{j\ell}}g_{p\ov{q}} \\
J_2 =&   g^{\ov{j}i}\ti{g}^{\ov{\ell}k}\lf(\ti{\nabla}_i\ov{\ti{T}^q_{j\ell}}-\ti{R}_{i\ov{\ell}p\ov{j}}\ti{g}^{\ov{q}p}\rt)g_{k\ov{q}}\\
J_3 =& -   \bigg[  g^{\ov{j}i}\left(  \ti{\nabla}_i\ov{\ti{T}^{\ell}_{j\ell}}   \right) +
 \left(\ti{\nabla}_{\ov{\ell}} \ti{T}^p_{ik}\right)g^{\ov{j}i}\ti{g}^{\ov{\ell}k} \ti{g}_{p \ov{j}}- \ov{\ti{T}^q_{j\ell}} \ti{T}^p_{ik}g^{\ov{j}i}\ti{g}^{\ov{\ell}k} \ti{g}_{p \ov{q}}
\bigg].
\ea
$$
Using Lemma \ref{keylemma}, part (\romannumeral1) established right now and the Cauchy-Schwarz inequality, we get (\ref{nice1}).

The inequality (\ref{nice2}) comes from a parabolic Schwarz Lemma argument as in \cite{ftwz,songtian,yau}. A key difference is that we do not have a global holomorphic map from $M_K$ to a lower dimensional complex manifold. Fortunately, we have a locally defined holomorphic map $f$ from a holomorphic chart in $M_K$ to the cross product $\h^{m-1}$ of $m-1$ upper half planes with the property that $\al=f^{\ast}\omega_{\h}$, where
$\omega_{\h}$ is the multiple of the product of Poincar\'{e} metrics
$$
\omega_{\h}=\mn\sum\limits_{i=1}^{m-1}\frac{\md z_i\wedge\md\ov{z_i}}{4y_{i}^2}
$$
on $\h^{m-1}$. Since the parabolic Schwarz Lemma calculation is completely local, we can deduce the inequality (\ref{nice2}) exactly as in \cite {twymathann}.

Now we turn to the estimate of the bound of Chern scalar curvature $R$. First note that the minimum principle and the evolution equation of $R$
$$
\frac{\de R}{\de t}=\Delta R+|\mathrm{Ric}|^2+R
$$
imply the lower bound $R\geq -C$ directly. For the upper bound of $R$, we consider the quantity
$$
u:=\varphi+\dot{\varphi}
$$
with the property that $-\Delta u=R+\tr_{\omega}\al\geq R$. We want to bound $-\Delta u$ from above by $Ce^{t/2}$.
Using Cheng-Yau type argument in \cite{chengyau} (cf. \cite{sesumtian,songtian}) and applying the maximum principle to
$$
Q_1:=\frac{|\n u|^2}{A-u}+C_1(\tr_{\omega}\al+C_0\tr_{\ti\omega}\omega)
$$
for $A$ and $C_1$ chosen sufficiently large, we can deduce the estimate
$$
|\n u|_{g}^2\leq Ce^{t/2},
$$
exactly as in \cite[Proposition 6.3]{twymathann} (replacing $\omega_{S}$ with $\al$ wherever it occurs). A direct calculation gives
\begin{equation}\label{nice4}
\left( \ddt{} - \Delta \right)|\nabla u|^2_g\leq -\frac{1}{2}|\nabla \ov{\nabla} u|^2_g-|\nabla \nabla u|^2_g+
|\nabla \tr_{\omega}{\al}|^2_g+|\ti{\nabla}g|^2_g+Ce^t.
\end{equation}
On the other hand, we have
\begin{equation}\label{nice5}
\left( \ddt{} - \Delta \right)(-\Delta u)\leq 2|\nabla \ov{\nabla} u|^2_g-\Delta u+Ce^{t/2} +|\hat{\nabla} g|^2_g-C^{-1}|\nabla \tr_{\omega}{\al}|^2_g.
\end{equation}
Combining (\ref{nice3}), (\ref{nice4}) and (\ref{nice5}) implies, for $C_1$ large,
$$
\left( \ddt{} - \Delta \right)\lf(-\Delta u+6|\nabla u|^2_g+C_1(\tr_{\omega}{\al}+C_0\tr_{\ti{\omega}}{\omega})\rt)\leq -|\nabla \ov{\nabla} u|^2_g-\Delta u+Ce^t,
$$
and it follows from the maximum principle that $-\Delta u\leq Ce^{t/2}$, giving the upper bounded of the Chern scalar curvature
$$
R\leq Ce^{t/2}.
$$
Next, using the discussion in \cite[Lemma 6.4]{twymathann} (cf. \cite{swlecturenotes}), we can deduce the bound (\romannumeral3) on $\dot{\varphi}$  follows from the bounds on $R$, Lemma \ref{ordinarylemma} and the evolution equation
$$
\ddt\dot{\varphi}=-R-(m-1)-\dot{\varphi}.
$$
For (\romannumeral4), we first obtain, as in \cite[Lemma 7.3]{twymathann} (replacing $\omega_S$ with $\al$, whenever it occurs),
$$
\lf(\ddt-\Delta\rt)\tr_{\omega}\ti\omega\leq Ce^{t/2}-C^{-1}|\ti\n g|_g^2.
$$
Combining this and the bounds of $\varphi$ and $\dot{\varphi}$, we consider the quantity
$$
e^{\varepsilon t}(\tr_{\omega}\ti\omega-m)-e^{\delta t}\varphi
$$
for $0<\varepsilon<1/4$ and $\delta,\,\delta'>0$ chosen carefully. The maximum principle discussion of \cite[Proposition 7.3]{twymathann} (replacing $\tr_{\omega}\ti\omega-2$ with $\tr_{\omega}\ti\omega-m$, whenever it occurs) gives
\be\label{nice6}
\tr_{\omega}\ti\omega-m\leq Ce^{-\varepsilon t}.
\ee
On the other hand, there exists a uniform $T_I>0$ depending only on the initial data of $M_K$ such that there holds, for $t\geq T_I$,
\begin{align}
\frac{\ti\omega^m}{\omega^m}
=&\frac{e^t\ti\omega^m}{\Omega}\frac{\Omega}{e^t\omega^m}
=\frac{e^t\ti\omega^m}{\Omega}e^{-\varphi-\dot{\varphi}}\label{nice7}\\
=&e^{-\varphi-\dot{\varphi}}(1+O(e^{-t}))\nonumber\\
\geq&e^{-\varphi-\dot{\varphi}}-Ce^{-t}
\geq 1-C'e^{-\sigma t},\nonumber
\end{align}
where we use (\ref{volequa}), Lemma \ref{ordinarylemma} and the bound of $\varphi+\dot{\varphi}$ in part (\romannumeral3). From (\ref{nice6}) and (\ref{nice7}) we can apply \cite[Lemma 2.6]{twypreprint} (choose coordinates such that $\omega$ is the identity and $\ti\omega$ is given by matrix $A$ and take $\varepsilon=\sigma$ in (\ref{nice6}) and (\ref{nice7})) to get
$$
\|\omega-\ti\omega\|_{C^{0}(M_K,\,\omega)}\leq C_{\sigma}e^{-\sigma t/2}.
$$
Noting $\omega\leq C\omega_{0}$, we can deduce part (\romannumeral4) for $t\geq T_I$. Then we can modify the uniform constant such that part (\romannumeral4) holds for all $t\geq 0$.
\end{proof}
Using the estimate in Part (\romannumeral1) of Theorem \ref{thmuse}, we can get the Gromov-Hausdorff convergence of $\omega(t)$ in Theorem \ref{thm1}.
\begin{proof}[Proof of Theorem \ref{thm1}]
Note that $\omega_{\infty}=\al$ represents $-c_{1}^{\mathrm{BC}}(M_K)$, and by definition of the reference metric $\ti\omega(t)\longrightarrow \al$ uniformly and exponentially fast as $t\longrightarrow \infty$. Part (\romannumeral4) in Theorem \ref{thmuse} implies that the same convergence holds for $\omega(t)$.

Now we turn to determining the Gromov-Hausdorff limit of $(M_K,\omega(t))$.
Call
$$
F=p:\;M_{K}\longrightarrow \torus^{m-1}
$$
the projection map and denote by $T_a=F^{-1}(a)$ the $\torus^{m+1}$-fiber over $a\in \torus^{m-1}$.  Fix $\varepsilon>0$ and let $L_t$ be the length of a curve in $M_K$ measured with respect to metric $\omega(t)$ and $\md_t$ be the induced distance function on $M_K$.  Also denote by $L_{\infty}$ and $\md_{\infty}$ the length and distance function of the degenerate metric $\al$ on $M_K$ and by $L$ and $\md$ the length and distance function of the flat Riemannian  metric $g$ defined by (\ref{metrictorus}) on $\torus^{m-1}$. Define
$$
G:\;\torus^{m-1}\longrightarrow M_{K}
$$
by mapping every point $a\in\torus^{m-1}$ to some chosen point in $M_K$ on the fiber $T_a$. Note that the map $G$ will in general be discontinuous.

Clearly we have $F\circ G=\mathrm{Id}$, while $G\circ F$ is a fiber-preserving discontinuous map of $M_{K}$. In particular for any $a\in \torus^{m-1}$ we have trivially
\be\label{gh1}
\md (a,\,F\circ G(a))=0.
\ee
Since $\F$ is the kernel of $\al$ and each leaf of $\F$ is dense in a $\torus^{m+1}$-fiber, we conclude that
$\md_{\infty}(x,\,y)=0$ for all $x,\,y\in M_K$ with $F(x)=F(y)$.
Therefore, combining Part (\romannumeral1) in Theorem \ref{thmuse} and uniform convergence of $\tilde{\omega}(t)$ to $\al$ as $t\longrightarrow \infty$ implies that for any $x\in M_K$ and for all $t$ large enough we have
\be\label{gh2}
\md_{t}(x,\,G\circ F(x))\leq \varepsilon.
\ee
Now take any two points $x,\,y\in M_K$ and let $\gamma$ be a curve joining $x$ to $y$ with $L_t(\gamma)=\md_{t}(x,\,y)$. Then $F(\gamma)$ is a path in $\torus^{m-1}$ between $F(x)$ and $F(y)$. We claim that
$$
L(F(\gamma))\leq L_{\infty}(\gamma).
$$
Indeed, for any tangent vector $V$ on $M_K$, we can write it locally as
$$
V=\sum\limits_{i=1}^m\lf(X^i\de_{x_i}+Y^i\de_{y_i}\rt)
$$
and from  the definitions of $\omega_{\infty}$ and $g$ we see that
$$
|F_{\ast}V|_{g}^2=\sum\limits_{i=1}^{m-1}\frac{(Y^{i})^2}{2y_{i}^2}\leq\sum\limits_{i=1}^{m-1}\frac{(X^{i})^2+(Y^{i})^2}{2y_{i}^2}=|V|_{\omega_{\infty}}^2,
$$
where for convenience we use the local coordinates $y_1,\cdots,y_{m-1}$ on $\torus^{m-1}$ and write the flat Riemannian metric $g$ on $\torus^{m-1}$  of form in (\ref{metricuse1}).
Choosing $V=\dot{\gamma}$ implies the claim.
Therefore, given two points $x,\,y\in M_K$, letting $\gamma$ be a minimizing geodesic for the metric $\omega(t)$ joining them, we can get
\be\label{gh3}
\md(F(x),\,F(y))\leq L(F(\gamma))\leq L_{\infty}(\gamma)\leq L_{t}(\gamma)+\varepsilon=\md_{t}(x,\,y)+\varepsilon,
\ee
for all $t$ large. Obviously this also implies that
\be\label{gh4}
\md(a,\,b)=\md(F\circ G(a),\,F\circ G(b))\leq \md_{t}(G(a),\,G(b))+\varepsilon,
\ee
for all $a,\,b\in \torus^{m-1}$ and all $t$ large.

Lastly, given $x,\,y\in M_K$, let $\gamma$ be a minimizing geodesic in $\torus^{m-1}$ connecting $F(x)$ and $F(y)$, and let $\ti\gamma$ be a lift of the curve $\gamma$ starting at $x$, that is, $\ti\gamma$ is a curve in $M_K$ with $F(\ti\gamma)=\gamma$ and initial point $x$. This lift can always be constructed because $F=p$ is the bundle projection map. We then concatenate $\ti\gamma$ with a curve $\ti\gamma_{1}$ contained in the fiber $T_{F(y)}$ joining the endpoint of $\ti\gamma$ with $y$, and obtain a curve $\hat{\gamma}$ in $M_K$ joining $x$ and $y$.
We claim that
$$
L_{\infty}(\ti\gamma)=L(\gamma)=\md(F(x),\,F(y)).
$$
In fact we can construct a lift
$\tilde{\gamma}$ such that in local coordinates we have $\dot{\tilde{\gamma}}=\sum\limits_{i=1}^{m-1}Y^i(t)\de_{y_i}$,  as in \cite{twcomplexsurface}.
So we can obtain
$$
|\dot{\gamma}|_{g}^2=|F_{\ast}\dot{\ti\gamma}|_{g}^2=\sum\limits_{i=1}^{m-1}\frac{(Y^{i}(t))^2}{2y_{i}^2}=|\dot{\ti\gamma}|_{\omega_{\infty}}^2,
$$
as required. Therefore, we can conclude that
\be\label{gh5}
\md_{t}(x,\,y)\leq L_{t}(\hat{\gamma})=L_{t}(\ti\gamma)+L_t(\ti\gamma_1)\leq L_{\infty}(\ti\gamma)+2\varepsilon=\md(F(x),\,F(y))+2\varepsilon,
\ee
for all $t$ large. This also implies
\be\label{gh6}
\md_{t}(G(a),\,G(b))\leq \md(F\circ G(a),\,F\circ G(b))+2\varepsilon=\md(a,\,b)+2\varepsilon
\ee
for all $a,\,b\in \torus^{m-1}$.
Combining (\ref{gh1}), (\ref{gh2}), (\ref{gh3}), (\ref{gh4}), (\ref{gh5}) and (\ref{gh6}) implies that $(M_K,\,\omega(t))$ converges to $(\torus^{m-1},\,g)$ in the Gromov-Hausdorff sense.
\end{proof}
Now we turn to Theorem \ref{thm3}. With the special reference metric
$$
\ti \omega=e^{-t}\omega_{\mathrm{OT}}+(1-e^{-t})\al,
$$
we can get better estimates than the ones in Lemma \ref{keylemma} as follows.
\begin{lem}\label{specialkeylemma}
There exists a uniform constant $C$ such that
\begin{enumerate}\label{keylemma2}
  \item[(\romannumeral1)]$|\ti T|_{\ti g}\leq C$.
  \item[(\romannumeral2)]$|\ov{\de}\ti T|_{\ti g}+|\ti \n\ti T|_{\ti g}+|\widetilde{\mathrm{Rm}}|_{\ti g}\leq C$.
   \item[(\romannumeral3)]$|\ov{\ti\n\ti\n}\ti T|_{\ti g}+|\ti \n\ov{\ti \n}\ti T|_{\ti g}+|\ti \n\widetilde{\mathrm{Rm}}|_{\ti g}\leq C$.
\end{enumerate}
\end{lem}
\begin{proof}
For this new reference metric $\ti g$, we also have
\begin{align}
\ti g_{k\ov{k}}
=&e^{-t}(g_{\mathrm{OT}})_{k\ov{k}}+(1-e^{-t})\al_{k\ov{k}}=O(1),\quad 1\leq k\leq m-1,\no\\
\ti g_{k\ov{\ell}}
=&e^{-t}(g_{\mathrm{OT}})_{k\ov{\ell}}=O(e^{-t}),\quad \mbox{otherwise},\nonumber
\end{align}
and that all the components of the inverse metric matrix  $(\ti g^{\ov{\ell}k})$ are bounded by $O(1)$ except that $\ti g^{\ov{m}m}
$ is bounded by $O(e^{t})$. Part (\romannumeral1) was proved in Lemma \ref{keylemma}.

Since $|\ti \Gamma_{ik}^p|_{\ti g}^2\leq C$, that is, the $\ti g$ norm of the first order derivatives of $\ti g$ is bounded, using the Cauchy-Schwarz inequality, to complete  the proof of the rest of the lemma, we just need to bound the $\ti g$ norm of the second and third order derivatives of the new reference metric $\ti g$.
Since
\be\label{use3}
\de_{m}\ti g_{k\ov{\ell}}=\de_{\ov{m}}\ti g_{k\ov{\ell}}=0,\quad 1\leq  k,\,\ell\leq m
\ee
and
\be\label{use4}
\de_{i}\ti g_{m\ov{m}}=\de_{\ov{i}}\ti g_{m\ov{m}}=O(e^{-t}),\quad \ti g_{i\ov{m}}=\ti g_{m\ov{i}}=0,\quad 1\leq i\leq m-1,
\ee
we can deduce
\begin{align*}
|\de_{i}\de_{\ov{j}}\ti g_{k\ov{\ell}}|_{\ti g}^2
=&\sum\limits_{i,j,k,\ell,r,s,p,q=1}^{m-1}\lf(\de_{i}\de_{\ov{j}}\ti g_{k\ov{\ell}}\rt)\lf(\de_{s}\de_{\ov{r}}\ti g_{q\ov{p}}\rt)
\ti g^{\ov{r}i}\ti g^{\ov{j}s}\ti g^{\ov{p}k}\ti g^{\ov{\ell}q}\\
&+\sum\limits_{i,j,r,s=1}^{m-1}\lf(\de_{i}\de_{\ov{j}}\ti g_{m\ov{m}}\rt)\lf(\de_{s}\de_{\ov{r}}\ti g_{m\ov{m}}\rt)
\ti g^{\ov{r}i}\ti g^{\ov{j}s}\ti g^{\ov{m}m}\ti g^{\ov{m}m}\leq C
\end{align*}
and
\begin{align*}
|\de_{a}\de_{i}\de_{\ov{j}}\ti g_{k\ov{\ell}}|_{\ti g}^2
=&\sum\limits_{a,b,i,j,k,\ell,r,s,p,q=1}^{m-1}\lf(\de_{a}\de_{i}\de_{\ov{j}}\ti g_{k\ov{\ell}}\rt)\lf(\de_{\ov{b}}\de_{s}\de_{\ov{r}}\ti g_{q\ov{p}}\rt)
\ti g^{\ov{b}a}\ti g^{\ov{r}i}\ti g^{\ov{j}s}\ti g^{\ov{p}k}\ti g^{\ov{\ell}q}\\
&+\sum\limits_{a,b,i,j,r,s=1}^{m-1}\lf(\de_{a}\de_{i}\de_{\ov{j}}\ti g_{m\ov{m}}\rt)\lf(\de_{\ov{b}}\de_{s}\de_{\ov{r}}\ti g_{m\ov{m}}\rt)
\ti g^{\ov{b}a}\ti g^{\ov{r}i}\ti g^{\ov{j}s}\ti g^{\ov{m}m}\ti g^{\ov{m}m}\leq C,
\end{align*}
as required.
\end{proof}
The estimates in Theorem \ref{thmuse} imply that the special reference metric $\ti g$ is equivalent to the solution metric $g$ uniformly. The other key ingredient of the proof of Theorem \ref{thm3} is the following Calabi-type ``third order" estimate.
\begin{prop}
For the normalized Chern-Ricci flow (\ref{crf}) with the reference metric
$$
\ti \omega=e^{-t}\omega_{\mathrm{OT}}+(1-e^{-t})\al,
$$
we have
\be\label{3rd}
|\ti \n g|_{\ti g}\leq C.
\ee
\end{prop}
\begin{proof}
This Calabi-type estimate is very similar to the ones established in \cite{shermanweinkove} and \cite[Section 8]{twymathann} (see also \cite{phongsesumsturm}), so we give only a brief outline, pointing out the main differences.
We consider the quantity $S:=|\ti\n g|_{g}^2=|\Psi|_{g}^2$, where $\Psi_{ij}^k=\Gamma_{ij}^k-\ti\Gamma_{ij}^k$ as in \cite{shermanweinkove}. The quantity $S$ is equivalent to $|\ti \n g|_{\ti g}^2$ because $g$ is equivalent to $\ti g$.

Compared to the setup in \cite{shermanweinkove}, here we consider the normalized Chern-Ricci flow and the reference metric $\ti g$ now depends on time $t$, while the reference metric $\hat{g}$ in \cite{shermanweinkove} is fixed. Combining these differences and the calculation in \cite{shermanweinkove}, we observe that there is one new term in the evolution of the quantity $S$ of the form
\be\label{newterm}
-2\mathrm{Re}\lf(g_{i\ov{r}}g^{\ov{u}j}g^{\ov{v}k}\ti g^{\ov{q}i}\ti\n_{j}\al_{k\ov{q}}\ov{\Psi_{uv}^r}\rt).
\ee
We claim that
\be\label{nice8}
|\ti\n \al|_{\ti g}\leq C.
\ee
Indeed, since
$
|\ti\Gamma_{ij}^k|_{\ti g}^2\leq C
$
and
$$
|\al|_{\ti g}^2=\sum\limits_{i,j,k,\ell=1}^{m-1}\ti g^{\ov{j}i}\ti g^{\ov{\ell}k}\al_{i\ov{\ell}}\al_{k\ov{j}}
=\sum\limits_{i,j=1}^{m-1}\ti g^{\ov{j}i}\ti g^{\ov{i}j}\al_{i\ov{i}}\al_{j\ov{j}}\leq C,
$$
from the Cauchy-Schwarz inequality, to prove (\ref{nice8}), it is enough to bound $\de_{i}\al_{j\ov{\ell}}$.
Noting that
$$
\de_{m}\al_{i\ov{j}}=\de_{\ov{m}}\al_{i\ov{j}}=\al_{m\ov{\ell}}=\al_{\ell \ov{m}}=0, \quad 1\leq i,\,j \leq m-1,\;1\leq \ell\leq m,
$$
we have
\begin{align*}
|\de_{i}\al_{j\ov{\ell}}|_{\ti g}^2
=&\sum\limits_{i,j,k,\ell,p,q=1}^{m-1}\ti g^{\ov{j}i}\ti g^{\ov{q}p}\ti g^{\ov{\ell}k}(\ti\de_{i}\al_{k\ov{q}})(\ti\de_{\ov{j}}\al_{p\ov{\ell}})
=\sum\limits_{i,j,k,\ell=1}^{m-1}\ti g^{\ov{j}i}\ti g^{\ov{k}\ell}\ti g^{\ov{\ell}k}(\ti\de_{i}\al_{k\ov{k}})(\ti\de_{\ov{j}}\al_{\ell\ov{\ell}})\leq C,\\
\end{align*}
as required.
Therefore, the new term (\ref{newterm}) is of the order $O(\sqrt{S})$ and harmless.
Combining  \cite[Remark 3.1]{shermanweinkove} and the estimates in Lemma \ref{keylemma2} gives the bound $S$ and hence  we can obtain (\ref{3rd}).
\end{proof}
\begin{proof}[Proof of Theorem \ref{thm3}]
We claim that
\be\label{nice9}
|\ti\Gamma-\Gamma_{\mathrm{OT}}|_{g_{\mathrm{OT}}}\leq C,
\ee
where $\Gamma_{\mathrm{OT}}$ is the Christoffel symbols of $g_{\mathrm{OT}}$.
Indeed, we just need to bound the $g_{\mathrm{OT}}$ norm of $\ti\Gamma$. Thanks to (\ref{use3}) and (\ref{use4}), we have
\begin{align*}
|\ti\Gamma|_{g_{\mathrm{OT}}}^2
=&(g_{\mathrm{OT}})_{k\ov{q}}(g_{\mathrm{OT}})^{\ov{r}i}(g_{\mathrm{OT}})^{\ov{s}j}\ti g^{\ov{\ell}k}\ti g^{\ov{q}p}\lf(\de_{i}\ti g_{j\ov{\ell}}\rt)\lf(\de_{\ov{r}}\ti g_{p\ov{s}}\rt)\\
=&(g_{\mathrm{OT}})_{m\ov{m}}(g_{\mathrm{OT}})^{\ov{r}i}(g_{\mathrm{OT}})^{\ov{s}j}\ti g^{\ov{m}m}\ti g^{\ov{m}m}\lf(\de_{i}\ti g_{j\ov{m}}\rt)\lf(\de_{\ov{r}}\ti g_{m\ov{s}}\rt)\\
&+\mbox{other terms bounded by constant}\\
=&(g_{\mathrm{OT}})_{m\ov{m}}(g_{\mathrm{OT}})^{\ov{r}i}(g_{\mathrm{OT}})^{\ov{m}m}\ti g^{\ov{m}m}\ti g^{\ov{m}m}\lf(\de_{i}\ti g_{m\ov{m}}\rt)\lf(\de_{\ov{r}}\ti g_{m\ov{m}}\rt)\\
&+\mbox{other terms bounded by constant}\leq C,
\end{align*}
as required. Now we use (\ref{3rd}), (\ref{nice9}) and the fact $\ti g\leq Cg_{\mathrm{OT}}$ to deduce
$$
|\n_{\mathrm{OT}} g|_{g_{\mathrm{OT}}}\leq |\ti\n g|_{g_{\mathrm{OT}}}+C\leq |\ti\n g|_{\ti g}+C\leq C,
$$
which completes the proof of Theorem \ref{thm3}.
\end{proof}


{\small\begin{flushleft}
Tao Zheng\\
School of Mathematics and Statistics\\
Beijing Institute of Technology\\
Beijing 100081\\
the People's Republic of China\\
E-mail: zhengtao08@amss.ac.cn
\end{flushleft}}

\end{document}